\documentclass[11pt]{amsart}
\usepackage[colorlinks,citecolor=red,dvipdfm]{hyperref}

\newtheorem{theorem}{Theorem}[section]
\newtheorem{lemma}[theorem]{Lemma}

\theoremstyle{definition}

\newtheorem{corollary}[theorem]{Corollary}
\newtheorem{remark}[theorem]{Remark}

\theoremstyle{remark}

\newcommand{\bs}{\begin{split}}
\newcommand{\es}{\begin{split}}

\newcommand{\be}{\begin{equation}}
\newcommand{\ee}{\end{equation}}

\numberwithin{equation}{section}

\begin{document}

\title{The Hirzebruch $\chi_y$-genus and Poincar\'{e} polynomial revisited}

\author{Ping Li}
\address{Department of Mathematics,
Tongji University, Shanghai 200092, China}

\email{pingli@tongji.edu.cn,\qquad  pinglimath@gmail.com}
\thanks{The author was supported by National
Natural Science Foundation of China (Grant No. 11471247) and the
Fundamental Research Funds for the Central Universities.}




\begin{abstract}
The Hirzebruch $\chi_y$-genus and Poincar\'{e} polynomial share some
similar features. In this article we investigate two of their
similar features simultaneously. Through this process we shall
derive several new results as well as reprove and improve some known
results.
\end{abstract}

\maketitle

\tableofcontents

\section{Introduction}
Let $X$ be a $k$-dimensional closed orientable manifold and $b_i(X)$
its $i$-th Betti number. The Poincar\'{e} polynomial of $X$, denoted
by $P_y(X)$, is by definition the generating function of its Betti
numbers:
$$P_y(X):=\sum_{i=0}^{k}b_i(X)\cdot y^i.$$
$P_y(X)$ satisfies a basic relation $P_y(X)=y^k\cdot P_{y^{-1}}(X)$
which is nothing but a reformulation of the Poincar\'{e} dualities
$b_i=b_{k-i}$ ($0\leq i\leq k$). When evaluated at $y=-1$,
$P_y(X)\big|_{y=-1}$ gives the most important (combinatorial)
invariant: the Euler characteristic. In principle, to determine the
Poincar\'{e} polynomial of a manifold is equivalent to knowing all
its Betti numbers.

Now we turn to the definition of the Hirzebruch $\chi_y$-genus,
which was first introduced by Hirzebruch in his seminal book
\cite{Hi} for projective manifolds, and can be computed by the
celebrated Hirzebruch-Riemann-Roch formula also established in
\cite{Hi}. Later on the discovery of the Atiyah-Singer index theorem
tells us that it still holds for almost-complex manifolds. To be
more precise, let $(M^{2n},J)$ be a compact almost-complex manifold
with complex dimension $n$ and an almost-complex structure $J$. As
usual we use $\bar{\partial}$ to denote the $d$-bar operator which
acts on the complex vector spaces $\Omega^{p,q}(M)$ ($0\leq p,q\leq
n$) of $(p,q)$-type differential forms on $(M^{2n},J)$ in the sense
of $J$ (\cite[p. 27]{We}). The choice of an almost Hermitian metric
on $(M^{2n},J)$ enables us to define the Hodge star operator $\ast$
and the formal adjoint
$\bar{\partial}^{\ast}=-\ast\bar{\partial}~\ast$ of the
$\bar{\partial}$-operator. Then for each $0\leq p\leq n$, we have
the following Dolbeault-type elliptic operator
\be\label{GDC}\bigoplus_{\textrm{$q$
even}}\Omega^{p,q}(M)\xrightarrow
{\bar{\partial}+\bar{\partial}^{\ast}} \bigoplus_{\textrm{$q$
odd}}\Omega^{p,q}(M),\ee whose index is denoted by $\chi^{p}(M)$ in
the notation of Hirzebruch in \cite{Hi}. The Hirzebruch
$\chi_{y}$-genus, denoted by $\chi_{y}(M)$, is the generating
function of these indices $\chi^p(M)$:
$$\chi_{y}(M):=\sum_{p=0}^{n}\chi^{p}(M)\cdot y^{p}.$$

The general form of the Hirzebruch-Riemann-Roch theorem, which is a
corollary of the Atiyah-Singer index theorem, allows us to compute
$\chi_y(M)$ in terms of the Chern numbers of $M$ as follows
\be\label{HRR}\chi_y(M)=\int_M\prod_{i=1}^n\frac{x_i(1+ye^{-x_i})}{1-e^{-x_i}},\ee
where $x_1,\ldots,x_n$ are formal Chern roots of $(M,J)$, i.e., the
$i$-th elementary symmetric polynomial of $x_1,\ldots,x_n$ is the
$i$-th Chern class of $(M,J)$. Similar to that of the Poincar\'{e}
polynomial, $\chi_y(M)$ also satisfies
$\chi_y(M)=(-y)^n\cdot\chi_{y^{-1}}(M)$ which are equivalent to the
relations $\chi^p=(-1)^n\chi^{n-p}$ and can be derived from
(\ref{HRR}). For three values of $y$, this $\chi_y$-genus is an
important invariant: $\chi_y(M)\big|_{y=-1}$ is the Euler
characteristic of $M$, $\chi_y(M)\big|_{y=0}$ is the Todd genus of
$M$, and $\chi_y(M)\big|_{y=1}$ is the signature of $M$.

 When
$J$ is integrable, i.e., $M$ is an $n$-dimensional compact complex
manifold, which is equivalent to the condition that
$\bar{\partial}^2\equiv0$, the two-step elliptic complex (\ref{GDC})
above has  the following resolution, which is the well-known
Dolbeault complex: \be\label{DC}0\rightarrow\Omega^{p,0}(M)
\xrightarrow{\bar{\partial}}\Omega^{p,1}(M)
\xrightarrow{\bar{\partial}}\cdots\xrightarrow{\bar{\partial}}\Omega^{p,n}(M)\rightarrow
0\ee and hence
 \be\label{Hodgenumber}\chi^{p}(M)=\sum_{q=0}^{n}(-1)
^{q}\text{dim}_{\mathbb{C}}H^{p,q}
_{\bar{\partial}}(M)=:\sum_{q=0}^{n}(-1)^{q}h^{p,q}(M).\ee Here
$h^{p,q}(M)$ are the corresponding Hodge numbers of $M$, which are
the complex dimensions of the corresponding Dolbeault cohomology
groups $H^{p,q} _{\bar{\partial}}(M)$. The famous Serre duality
(\cite[p. 102]{GH}) gives the relation $h^{p,q}=h^{n-p,n-q}$, which
can be used to give an alternative proof of the fact
$\chi^p=(-1)^n\chi^{n-p}$ in the case of $J$ being integrable:
$$\chi^{p}=\sum_{q=0}^{n}(-1)^{q}h^{p,q}=
\sum_{q=0}^{n}(-1)^{q}h^{n-p,n-q}=
(-1)^n\sum_{q=0}^{n}(-1)^{q}h^{n-p,q}=(-1)^n\chi^{n-p}.$$

The Poincar\'{e} polynomial and the Hirzebruch $\chi_y$-genus are
two fundamental mathematical objects and have been studied
intensively from various aspects. However, as far as the author
knows, there is no explicit investigation in the existing literature
towards \emph{direct} connections between Poincar\'{e} polynomial
and the Hirzebruch $\chi_y$-genus. Indeed they do share some
similarities. For instance, as we have mentioned, their coefficients
satisfy similar duality relations: $b_i=b_{k-i}$ and
$\chi^p=(-1)^n\chi^{n-p}$, and when evaluated at $y=-1$, both
$P_y(X)\big|_{y=-1}$ and $\chi_y(M)\big|_{y=-1}$ give the Euler
characteristic.

Over the past several years, the author gradually realized that
these two mathematical objects should share some more interesting
and deeper similarities in various senses. For instance, Thompson
noticed in \cite{Th} that
 the character of the natural $SL(2,\mathbb{C})$-representation on
any hyperK\"{a}hler manifold induced by the holomorphic two form is
essentially its $\chi_y$-genus. Inspired by this interesting
observation and keeping the similarity between the $\chi_y$-genus
and Poincar\'{e} polynomial in mind, the author showed in\cite{Li3}
that there exists an analogous result for any compact K\"{a}hler
manifold: the character of the natural
$SL(2,\mathbb{C})$-representation on any compact K\"{a}hler manifold
induced by the K\"{a}hler form is essentially its Poincar\'{e}
polynomial.

The main purpose of the current article is to strengthen this belief
from two aspects by investigating some properties simultaneously for
manifolds with some extra structures (K\"{a}hler structure,
hyperK\"{a}hler structure, symplectic structure etc.). Through this
process we can derive a number of nontrivial results. Among these
results, some have been known for some time by using somewhat
different methods while some should be new, at least to the author's
best knowledge. The author believes that the relationship between
Poincar\'{e} polynomial and the Hirzebruch $\chi_y$-genus deserves
more attention and there should exist deeper interactions between
them.

The rest of this article is arranged as follows. In Section
\ref{section2} we introduce and investigate the slightly modified
coefficients of the Taylor expansions of Poincar\'{e} polynomial and
the $\chi_y$-genus at $y=-1$. Section \ref{section3} is devoted to
some related applications of these coefficients to K\"{a}hler and
hyperK\"{a}hler manifolds. In the first two subsections of Section
\ref{section4} we recall two residue formulas related to the
$\chi_y$-genus and Poincar\'{e} polynomial. Then in the third
subsection, Section \ref{section4.3}, we provide some geometric and
topological obstructions to the existence of Hamiltonian torus
actions with isolated fixed points on compact symplectic manifolds.

\section{``$-1$"-phenomenon of the $\chi_y$-genus}\label{section2}
The material in this section is inspired by an interesting
phenomenon of the $\chi_y$-genus, which the author calls \emph{the
``$-1$"-phenomenon} and has been observed, implicitly or explicitly,
in several independent articles.

\subsection{``$-1$"-Phenomena}\label{section2.1}
The purpose of this subsection is to recall this ``$-1$"-phenomenon
for the Hirzebruch $\chi_y$-genus.

 As we have mentioned in the
introduction, when evaluated at $y=-1$, $\chi_y(M)\big|_{y=-1}$
gives the Euler characteristic, which is equal to the top Chern
number $c_n$ of $M$. Note that $\chi_y(M)\big|_{y=-1}$ is exactly
the constant term in the Taylor expansion of $\chi_y(M)$ at $y=-1$.
In fact, several independent articles (\cite{NR}, \cite{LW},
\cite{Sa1}), with different backgrounds, have observed that, when
expanding the right-hand side of (\ref{HRR}) at $y=-1$, its
corresponding coefficients can be expressed \emph{explicitly} in
terms of Chern numbers. More precisely, if we denote
$$\int_M\prod_{i=1}^n\frac{x_i(1+ye^{-x_i})}{1-e^{-x_i}}=:\sum_{i=0}^na_i(M)\cdot(y+1)^i,$$
then we have
$$a_0=c_{n},\qquad a_1= -\frac{1}{2}nc_{n},$$
$$ a_2=
\frac{1}{12}[\frac{n(3n-5)}{2}c_{n}+c_{1}c_{n-1}],$$
$$a_3=-\frac{1}{24}[\frac{n(n-2)(n-3)}{2}c_{n}+(n-2)c_{1}c_{n-1}],$$
\be\begin{split}
a_4=\frac{1}{5760}[&n(15n^{3}-150n^{2}+485n-502)c_{n}+
4(15n^{2}-85n+108)c_{1}c_{n-1}\\
&+8(c_{1}^{2}+3c_{2})c_{n-2}-
8(c_{1}^{3}-3c_{1}c_{2}+3c_{3})c_{n-3}], \end{split}\nonumber\ee
$$\cdots.$$

The derivations of $a_0$ and $a_1$ are easy. The calculation of
$a_2$ appears implicitly in \cite[p. 18]{NR} and \cite[Corollary
5.3.12]{Fu} and explicitly in \cite[p. 141-143]{LW}. Narasimhan and
Ramanan used $a_2$ to give a topological restriction on some moduli
spaces of stable vector bundles over Riemann surfaces. The primary
interest of \cite[Ch. 5]{Fu} is to interpret the Futaki invariant on
Fano manifolds as a special case of a family of integral invariants.
But in Corollary 5.3.12 Futaki also implicitly computed the
expression $a_2$. Libgober and Wood used $a_2$ to prove the
uniqueness of the complex structure on K\"{a}hler manifolds of
certain homotopy types \cite[Theorems 1 and 2]{LW}. Inspired by
\cite{NR}, Salamon applied $a_2$ (\cite[Corollary 3.4]{Sa1}) to
obtain a restriction on the Betti numbers of hyperK\"{a}hler
manifolds (\cite[Theorem 4.1]{Sa1}). In \cite{Hi2}, Hirzebruch
applied $a_1$, $a_2$ and $a_3$ to deduce a divisibility result on
the Euler number of almost-complex manifolds with $c_1=0$. The
expressions $a_3$ and $a_4$ are also included in \cite[p. 145]{Sa1}.

\subsection{Technical preliminaries}
The purpose of this subsection is to introduce and investigate some
numerical values $h(p^i)$ and $f(i)$ related to the $\chi_y$-genus
and Poincar\'{e} polynomial, which are slightly modified from the
coefficients of the Taylor expansion of them at $y=-1$.

For any $n$-dimensional compact complex manifold $M$, this
``$-1$"-phenomenon tells us that, via the H-R-R formula (\ref{HRR})
and the relation (\ref{Hodgenumber}), some linear combinations of
the Hodge numbers of $M$ can be expressed in terms of its Chern
numbers:
$$a_0=\sum_{p=0}^n(-1)^p\cdot\chi^p,$$
and for $i\geq 1$,
  \be\label{aihodege}\begin{split}
 a_i&=\frac{(-1)^i}{i!}\sum_{p=i}^n(-1)^{p}\cdot\chi^p\cdot p(p-1)\cdots(p-i+1)\\
&=\frac{(-1)^i}{i!}\sum_{p=0}^n(-1)^{p}\cdot\chi^p\cdot p(p-1)\cdots(p-i+1)\\
&=\frac{(-1)^i}{i!}\sum_{p,q=0}^n(-1)^{p+q}\cdot h^{p,q}\cdot
p(p-1)\cdots(p-i+1)
\end{split}\ee

For our later convenience, we define, for any polynomial $x=x(p,q)$,
\be\label{h(x)}h(x):=\sum_{p,q=0}^{n} (-1)^{p+q}\cdot h^{p,q}\cdot
x.\ee

Using this symbol we know that
$$a_0=h(1)=h(p^0)$$
and
$$a_i=\frac{(-1)^i}{i!}h\big(p(p-1)\cdots(p-i+1)\big)\qquad\text{for}\qquad
i\geq 1.$$

 However, in order to reveal this
``$-1$"-phenomenon more efficiently, we would like to investigate
the slightly modified coefficients
$$h(p^i)=\sum_{p,q=0}^{n}
(-1)^{p+q}\cdot h^{p,q}\cdot p^i$$
 originating from $a_i$.

The following technical lemma tells us that the three sets
$\{a_i\}$, $\{h(p^i)\}$ and $\{h(p^{2i})\}$ contain the same
information.

\begin{lemma}\label{technicallemma1}
~
\begin{enumerate}
\item
 Any element in the set $\{a_0,a_1,\ldots,a_n\}$ can be
expressed
 in terms of the elements in the set $\{h(p^i),~0\leq i\leq n\}$ and vice visa.

 \item
 The first several explicit expressions of $h(p^i)$ in terms of Chern numbers are given below:
 $$h(1)=c_n,\qquad h(p^1)=\frac{n}{2}c_n,$$
 $$ h(p^2)=\frac{n(3n+1)}{12}c_n+\frac{1}{6}c_1c_{n-1},\qquad
  h(p^3)=\frac{n^2(n+1)}{8}c_n+\frac{n}{4}c_1c_{n-1},$$
  \be \begin{split}
  h(p^4)=&\frac{n(15n^3+30n^2+5n-2)}{240}c_n+\frac{15n^2+5n-2}{60}c_1c_{n-1}
  \\
  &+\frac{(c_1^2+3c_2)c_{n-2}}{30}-
  \frac{(c_1^3-3c_1c_2+3c_3)c_{n-3}}{30},\end{split}\nonumber\ee
  $$\cdots.$$

\item
Each $h(p^{2i+1})$ can be expressed by the even powers $h(1),
h(p^2),\cdots,h(p^{2i})$. This means the set $\{h(p^i),~0\leq i\leq
n\}$, as well as $\{a_i\}$ contains the same information as that in
the set $\{h(p^{2i}),~0\leq i\leq[\frac{n}{2}]\}$. So this
``$-1$"-phenomenon gives us $[\frac{n}{2}]+1$ independent relations
on the Chern numbers.
\end{enumerate}
\end{lemma}

\begin{proof}
$(1)$ is an elementary linear algebra exercise. The derivations of
$h(p^i)$ are also direct via the concrete expressions of $a_i$ in
(\ref{aihodege}). For example,
$$h(p^2)=h\big[p(p-1)+p\big]=2!\cdot a_2-a_1=\frac{n(3n+1)}{12}c_n+\frac{1}{6}c_1c_{n-1},$$

$$h(p^3)=h\big[p(p-1)(p-2)+3p(p-1)+p\big]=-6a_3+6a_2-a_1=\cdots,$$
and so on.

$(3)$ is an application of the Serre dualities
$h^{p,q}=h^{n-p,n-q}$. Indeed,
 \be\begin{split}
h(p^{2i+1})&=\sum_{p,q}(-1)^{p+q}\cdot h^{p,q}\cdot p^{2i+1}\\
&=\frac{1}{2}\sum_{p,q}(-1)^{p+q}\cdot h^{p,q}
\cdot[p^{2i+1}+(n-p)^{2i+1}]\qquad(\text{by $h^{p,q}=h^{n-p,n-q}$})\\
&=\frac{1}{2}\sum_{p,q}(-1)^{p+q}\cdot h^{p,q}\cdot[
\sum_{j=1}^{2i+1}{2i+1 \choose j}n^j(-p)^{2i+1-j}]\\
&=\frac{1}{2}\sum_{p,q}(-1)^{p+q}\cdot h^{p,q}\cdot[(2i+1)np^{2i}+\cdots]\\
&=(i+\frac{1}{2})n\cdot h(p^{2i})+(\cdots),
\end{split}\nonumber\ee
where ($\cdots$) is a sum of the terms $h(p^j)$ with $j<2i$. Using
this formula repeatedly yields the fact that $h(p^{2i+1})$ can be
expressed by $h(p^{2j}),~0\leq j\leq i$.\end{proof}

We have an analogous result to Lemma \ref{technicallemma1} for the
Poincar\'{e} polynomial, which we record in the following lemma.

\begin{lemma}\label{technicallemma2}
~
\begin{enumerate}
\item
The coefficients of the Taylor expansion of $P_y(X)$ at $y=-1$ can
be expressed in terms of the elements in the set
$$\{f(i):=\sum_{p=0}^k(-1)^p\cdot b_p\cdot
p^{i}~|~0\leq i\leq k\}.$$

\item
 The information contained in the set above is the same as that in its subset
$$\{f(2i)~|~0\leq i\leq [\frac{k}{2}]\}.$$
\end{enumerate}
\end{lemma}

\begin{proof}
The proof of this lemma is almost identically the same as that in
Lemma \ref{technicallemma1}. We only need to note that, in the proof
of $(2)$ in this lemma, Poincar\'{e} dualities play the role of the
Serre dualities as in the proof of $(3)$ in Lemma
\ref{technicallemma1}.
\end{proof}

The expressions $f(2i)$ and $h(p^i)$ highlighted in Lemmas
\ref{technicallemma1} and \ref{technicallemma2} shall be
investigated intensively in the next section.

\section{Applications of ``$-1$"-phenomenon to K\"{ahler} and hyperK\"{a}hler manifolds}\label{section3}
 With the preliminaries presented in the last
section at hand, we now give some applications via some suitable
manipulations on the above-introduced modified coefficients $h(p^i)$
and $f(2i)$.

In order to achieve our purpose, we need to bridge a link between
the Betti numbers $b_i$ and the Hodge numbers $h^{p,q}$ involving in
$f(2i)$ and $h(p^i)$ respectively. Recall that the Hodge theory
imposes intimate relations among the Betti numbers and the Hodge
numbers for compact K\"{a}hler manifolds. So from now on we assume
throughout this section that $M$ be a complex $n$-dimensional
compact connected K\"{a}hler manifold. Then its Betti numbers $b_i$
and Hodge numbers $h^{p,q}$ satisfy the following well-known
relations (\cite[p. 116]{GH}):
\be\label{Hodecom}b_i=\sum_{p+q=i}h^{p,q},\qquad h^{p,p}\geq
1,\qquad h^{p,q}=h^{n-p,n-q}=h^{q,p},\qquad (0\leq i\leq 2n,~0\leq
p,q\leq n).\ee The first equality is a consequence of the Hodge
decomposition theorem, the second one is due to the fact that any
$p$-th power of the K\"{a}hler form represents a nonzero cohomology
class in $H^{p,p}(M)$, and the third one comes from the Serre
duality which we have mentioned in the introduction and the complex
conjugation.

Now it is time for us to illustrate some applications via comparing
the modified coefficients $f(2i)$ and $h(p^i)$ simultaneously.

Comparing $f(0)$ and $h(1)$ leads to the well-known fact for the
Euler characteristic:
$$f(0)=\sum_{i=0}^{2n}(-1)^i\cdot b_i\stackrel{(\ref{Hodecom})}{=}\sum_{p,q=0}^n(-1)^{p+q}
h^{p,q}=h(1)=c_n.$$

Now we consider $f(2)$:
 \be\label{2expansion}\begin{split}
f(2)=\sum_{i=0}^{2n}(-1)^i\cdot b_i\cdot i^2
&=\sum_{p,q=0}^n(-1)^{p+q}\cdot
h^{p,q}\cdot(p+q)^2\qquad(\text{by $(\ref{Hodecom})$})\\
&=h(p^2)+2h(pq)+h(q^2)\qquad (\text{by $(\ref{h(x)})$})\\
&=2h(p^2)+2h(pq).\qquad\text{(by $h^{p,q}=h^{q,p}$)}\end{split}\ee

We know through Lemma \ref{technicallemma1} that $h(p^2)$ can be
expressed in terms of the Chern numbers $c_n$ and $c_1c_{n-1}$. Thus
in order to obtain some nontrivial relations from
(\ref{2expansion}), we need to impose more restrictions on $h^{p,q}$
in order to deal with another term $h(pq)$. Here we impose two
different restrictions. The first one leads to the following result
due to Salamon which gives a restriction on the Betti numbers of
hyperK\"{a}hler manifolds (\cite[Theorem 4.1, Corollary 4.2]{Sa1}).

\begin{theorem}[Salamon]\label{theorem1}
Suppose $M$ is a compact K\"{a}hler manifold whose complex dimension
$n$ is even and Hodge numbers are
 ``invariant by mirror symmetry" in the sense that $h^{p,q}=h^{p,n-q}$. Then
 the Chern number $c_1c_{n-1}$ of $M$ can be expressed in terms of its Betti numbers:
 \be\label{restriction1}c_1c_{n-1}=\sum_{i=0}^{2n}
 (-1)^i\cdot b_i\cdot[3i^2-n(3n+\frac{1}{2})].\ee
In particular, this gives a restriction on the Betti numbers of a
 compact hyper-K\"{a}hler manifold whose complex dimension is $n$:
$$\sum_{i=0}^{2n}
 (-1)^i\cdot b_i\cdot[3i^2-n(3n+\frac{1}{2})]=0.$$

\end{theorem}
\begin{proof}
Under our assumptions we have \be
\begin{split}h(pq)&=\sum_{p,q=0}^n(-1)^{p+q}\cdot h^{p,q}\cdot
pq\\
&=\sum_{p,q=0}^n(-1)^{p+n-q}\cdot h^{p,q}\cdot p(n-q)~(\text{by
$h^{p,q}=h^{p,n-q}$})\\
&=\sum_{p,q=0}^n(-1)^{p+q}\cdot h^{p,q}\cdot p(n-q)~(\text{by $n$ is even})\\
&=n\cdot h(p)-h(pq).
\end{split}\nonumber\ee

Thus $h(pq)=\frac{n}{2}h(p)$, which, together with
(\ref{2expansion}), yields
\be\begin{split}\sum_{i=0}^{2n}(-1)^i\cdot b_i\cdot i^2&=
2h(p^2)+n\cdot h(p)\\
&=\frac{n(3n+1)}{6}c_n+\frac{1}{3}c_1c_{n-1}+\frac{n^2}{2}c_n
~(\text{by Lemma \ref{technicallemma1}})\\
&=\frac{1}{3}c_1c_{n-1}+n(n+\frac{1}{6})\sum_{i=0}^{2n}(-1)^i\cdot
b_i.~\big(\text{by $c_n=\sum_{i=0}^{2n}(-1)^i\cdot b_i$}\big)
\end{split}\nonumber\ee
Singling out the term $c_1c_{n-1}$ in the equality above leads to
(\ref{restriction1}).

Recall that a \emph{hyperK\"{a}hler manifold} is a compact
Riemannian manifold whose real dimension is divisible by $4$, say
$4m$, and holonomy group is contained in $\text{Sp}(m)$, which is a
higher-dimensional analogue to $K3$-surfaces. It is well-known that
a hyperK\"{a}hler manifold possesses a family of K\"{a}hler
structures parameterized by a $2$-dimensional sphere. Moreover, its
Hodge numbers are ``invariant by mirror symmetry" in the sense that
$h^{p,q}=h^{p,2m-q}$ and all its odd Chern classes $c_{2i+1}$
($0\leq i\leq m-1$) vanish in $H^{\ast}(M,\mathbb{R})$.
HyperK\"{a}hler manifolds form an important subclass in compact
K\"{a}hler manifolds. We refer the reader to \cite{Hitchin} and
\cite{Huy} for an account of them and their basic properties. So
hyperK\"{a}hler manifolds satisfy the conditions assumed in this
theorem, which gives the required restriction.
\end{proof}\bibliographystyle{amsplain}

We say a compact K\"{a}hler manifold $M$ are of \emph{pure type} if
its Hodge numbers satisfy $h^{p,q}=0$ whenever $p\neq q.$ Many
important compact K\"{a}hler manifolds are of pure type (cf. Remark
\ref{remark}). Our second restrction on $h^{p,q}$ involving the
notion of pure
 type leads to the following result,
 which is an improvement of the author's previous result (\cite[Theorem 1.3]{Li1}).

\begin{theorem}\label{theorem2}
The Chern number $c_1c_{n-1}$ of any $n$-dimensional compact
K\"{a}hler manifold $M$ has a lower bound in terms of its Betti
numbers as follows:
$$c_1c_{n-1}\geq\frac{1}{2}\big\{
\sum_{\text{$i$ even}}b_i[3i^2-n(3n+1)]- \sum_{\text{$i$
odd}}b_i[9i^2-n(3n+1)]\big\},$$ where the equality holds if and only
if $M$ is of pure type.
\end{theorem}

\begin{proof}
We first claim that for any $n$-dimensional compact K\"{a}hler
manifold we have \be\label{inequality}
4h(pq)\leq\sum_{i=0}^{2n}b_i\cdot i^2,\ee where the equality holds
if and only if $M$ is of pure type. Indeed, \be
h(pq)=\sum_{p,q=0}^n(-1)^{p+q}\cdot h^{p,q}\cdot
pq\leq\sum_{p,q=0}^n h^{p,q}\cdot
pq\leq\sum_{p,q=0}^nh^{p,q}(\frac{p+q}{2})^2\stackrel{(\ref{Hodecom})}{=}\frac{1}{4}\sum_{i=0}^{2n}b_i\cdot
i^2.\nonumber\ee

Note that in the formula above the equality in the second $``\leq"$
holds if and only if $h^{p,q}=0$ whenever $p\neq q$, which also
implies the validity of the equality in the first $``\leq"$.

 Thus (\ref{inequality}), together with (\ref{2expansion}), yields
\be\begin{split}\sum_{i=0}^{2n}(-1)^i\cdot b_i\cdot i^2&=
2h(p^2)+2\cdot h(pq)\\
&\leq\frac{n(3n+1)}{6}c_n+\frac{1}{3}c_1c_{n-1}+\frac{1}{2}\sum_{i=0}^{2n}b_i\cdot i^2\\
&=\frac{n(3n+1)}{6}\sum_{i=0}^{2n}(-1)^i\cdot
b_i+\frac{1}{3}c_1c_{n-1}+\frac{1}{2}\sum_{i=0}^{2n}b_i\cdot i^2.
\end{split}\nonumber\ee
Rewriting this inequality by singling out the term $c_1c_{n-1}$
leads to the desired one in Theorem \ref{theorem2}.
\end{proof}

\begin{remark}\label{remark}
~
\begin{enumerate}
\item
An additional assumption in \cite[Theorem 1.3]{Li1} that the odd
Betti numbers $b_{2i+1}$ all vanish has been removed.

\item
As we have remarked in \cite[Remark 1.2]{Li1}, many important
compact K\"{a}hler manifolds are of pure type. For instance, the
flag manifold $G/P$ (\cite[\S14.10]{BH}), where $G$ is a complex
semisimple linear algebraic group and $P$ is a parabolic subgroup,
the Fano contact manifolds (\cite[p. 118]{LS}), and the nonsingular
projective toric varieties (\cite[p. 106]{Fulton}).

\item
It is well-known that the simplest Chern number $c_n$ of any
$n$-dimensional compact almost-complex manifold is equal to the
alternating sum of its Betti numbers. Theorem \ref{theorem2}
illustrates an interesting phenomenon that, for any $n$-dimensional
compact K\"{a}hler manifold, the next-to-simplest Chern number
$c_1c_{n-1}$ can also be related to Betti numbers.
\end{enumerate}
\end{remark}

Recall that a \emph{Calabi-Yau} manifold (in the weak sense) is a
compact K\"{a}hler manifold whose first Chern class $c_1=0$ in
$H^2(M,\mathbb{R})$. Clearly all hyperK\"{a}hler manifolds are
Calabi-Yau manifolds. Theorem \ref{theorem2} has the following
interesting consequence on the Betti numbers of Calabi-Yau
manifolds, which, to the author's best knowledge, should be new.

\begin{corollary}
The Betti numbers of any complex $n$-dimensional compact K\"{a}hler
manifold whose Chern number $c_1c_{n-1}=0$ satisfy
$$\sum_{\text{$i$
odd}}b_i[9i^2-n(3n+1)]\geq\sum_{\text{$i$ even}}b_i[3i^2-n(3n+1)],$$
where the equality holds if and only if it is of pure type. In
particular, this inequality holds for Calabi-Yau manifolds.
\end{corollary}

\begin{remark}
This inequality is sharp in the sense that its equality can be
attained for some manifolds. For $n=2$ the most famous examples are
Enriques surfaces, which are of pure type and $c_1=0$ in
$H^2(M,\mathbb{R})$. For general $n$, recall that in Remark
\ref{remark} we commented that nonsingular projective toric
varieties are of pure type. Their Chern classes can be explicitly
described by their irreducible $T$-divisors (\cite[p. 109]{Fu}). So
we can choose appropriately a fan $\Delta$ such that the sum of the
irreducible $T$-divisors is zero. This means by \cite[p. 109,
Lemma]{Fu} that the first Chern class of the corresponding
nonsingular projective toric variety $X(\Delta)$ is zero.
\end{remark}

In the discussions above the simultaneous investigations on $f(2)$
and $h(p^2)$ have produced plentiful results. In particular, it
provides a lower bound for the next-to-simplest Chern number
$c_1c_{n-1}$ of $n$-dimensional compact K\"{a}hler manifolds in
terms of their Betti numbers in Theorem \ref{theorem2}. However, as
we have discussed in the proof of Theorem \ref{theorem1},
$n$-dimensional hyperK\"{a}hler manifolds have vanishing odd Chern
classes and so $c_1c_{n-1}=0$ automatically. This means that their
next-to-simplest Chern numbers are $c_2c_{n-2}$. A natural question
is whether or not we have a lower bound in terms of Betti numbers
for the Chern numbers $c_2c_{n-2}$ of hyperK\"{a}hler manifolds. The
answer is yes and the method is to continue our employment of $f(4)$
and $h(p^4)$ simultaneously. To be more precise, we have the
following result.

\begin{theorem}
Suppose $M$ is a hyperK\"{a}hler manifold of complex dimension $n$,
which is necessarily even by definition. Then $c_2c_{n-2}$, the
next-to-simplest Chern number of $M$, has the following lower bound
in terms of Betti numbers:
\be\label{inequality2}\begin{split} &c_2c_{n-2}\\
\geq&\frac{1}{24} \big\{\sum_{\text{$i$
even}}b_i[75i^4-n(75n^3+90n^2+5n-2)]-\sum_{\text{$i$
odd}}b_i[165i^4-n(75n^3+90n^2+5n-2)]\big\}.\end{split}\ee
\end{theorem}

However, this lower bound is \emph{not} sharp in the sense that the
equality cannot be attained. The reason will be clear in the process
of the following proof.

\begin{proof}
First we have \be\label{1}f(4)=\sum_{i}(-1)^i\cdot b_i\cdot i^4
=\sum_{p,q}(-1)^{p+q}\cdot
h^{p,q}\cdot(p+q)^4=2h(p^4)+8h(p^3q)+6h(p^2q^2).\ee Since the odd
Chern classes of $M$ vanish, the expression $h(p^4)$ in Lemma
\ref{technicallemma1} can be simplified to the following form
\be\label{2}h(p^4)=\frac{n(15n^3+30n^2+5n-2)}{240}c_n+\frac{1}{10}c_2c_{n-2}.\ee
The conditions that $n$ be even and $h^{p,q}=h^{p,n-q}$ can be
employed to deal with the term $h(p^3q)$: \be\begin{split}
h(p^3q)=\sum_{p,q}(-1)^{p+q}\cdot h^{p,q}\cdot
p^3q&=\sum_{p,q}(-1)^{p+n-q}\cdot h^{p,q}\cdot p^3(n-q)\\
&=\sum_{p,q}(-1)^{p+q}\cdot h^{p,q}\cdot p^3(n-q)\\
&=n\cdot h(p^3)-h(p^3q).\end{split}\nonumber\ee We thereby obtain
$h(p^3q)=\frac{n}{2}h(p^3)$. Combining this with the expression
$h(p^3)$ in Lemma \ref{technicallemma1} we have
\be\label{3}h(p^3q)=\frac{n^3(n+1)}{16}c_n.\ee
 At last we derive
an inequality for $h(p^2q^2)$, whose method is the same as that in
(\ref{inequality}). \be\label{4}h(p^2q^2)=\sum_{p,q}(-1)^{p+q}\cdot
h^{p,q}\cdot (pq)^2\leq\sum_{p,q}h^{p,q}\cdot (pq)^2\leq
\sum_{p,q}h^{p,q}\cdot (\frac{p+q}{2})^4=\frac{1}{16}\sum_ib_i\cdot
i^4.\ee Putting (\ref{1})-(\ref{4}) together and doing some
calculations we can obtain (\ref{inequality2}).

Now we explain why the equality in (\ref{inequality2}) cannot be
attained. Indeed, similar to the reason in (\ref{inequality}),
(\ref{4}) is an equality if and only if $M$ is of pure type:
$h^{p,q}=0$ whenever $p\neq q$. But this is \emph{not} compatible
with the additional symmetry $``h^{p,q}=h^{p,n-q}"$ for
hyperK\"{a}hler manifolds as, for instance, $h^{0,n}=h^{0,0}=1\neq
0$.
\end{proof}

A few more remarks are in order before we end this section. We also
know the explicit expressions for $h(p^5)$ and $h(p^6)$. Indeed
Libgober and Wood described a concrete algorithm to compute $h(p^i)$
for general $i$ (\cite[p. 144]{LW}). So in principle we can employ
these to deal with $f(6)$, $f(8)$ and so on. But when $i$ increases,
the expressions for $h(p^i)$ become more and more complicate. This
means their expressions are too complicated to formulate some
geometrically interesting consequences.

\section{Residue formulas and their applications}\label{section4}
The material in this section is inspired by the residue formulas for
the $\chi_y$-genus on almost-complex manifolds and for the
Poincar\'{e} polynomial on symplectic manifolds. We recall the
residue formulas for the $\chi_y$-genus and Poincar\'{e} polynomial
in Sections \ref{section4.1} and \ref{section4.2} respectively, and
give some related applications to symplectic geometry in Section
\ref{section4.3}.

When a smooth or an almost-complex manifold admits a
\emph{compatible} vector field or a compact Lie group action, the
philosophy of residue formula is to reduce the investigation of some
global invariants on this manifold to the consideration of the local
information around the zero point set of this vector field or the
fixed point set of this group action. Here by ``\emph{compatible}"
we mean that the one-parameter group action induced by the vector
field or the Lie group action preserves the smooth or almost-complex
structure. What we are concerned with in this section are vector
fields on almost-complex manifolds with isolated zero points and
Hamiltonian torus actions on symplectic manifolds with isolated
fixed points, which we shall discuss respectively in what follows.

\subsection{Residue formula for the $\chi_y$-genus}\label{section4.1}
The purpose of this subsection is to review a residue formula for
the $\chi_y$-genus.

Suppose $(M,g,J)$ is a compact connected almost-Hermitian manifold
with complex dimension $n$. This means $J$ is an almost-complex
structure and $g$ an almost-Hermitian metric, i.e., a Riemannian
metric which is $J$-invariant. Now suppose we have a smooth vector
field $A$ on $(M,g,J)$ preserving the metric $g$ and the
almost-complex structure $J$ such that $\text{zero}(A)$, the zero
point set of $A$, is isolated (but nonempty). Let
$P\in\text{zero}(A)$ be an arbitrary isolated zero point. Then
$T_P$, the tangent space to $(M,J)$ at $P$, is an $n$-dimensional
complex vector space equipped with an Hermitian inner product
induced by $J$. Since $A$ preserves the Hermitian metric, $A$
induces a skew-Hermitian transformation on $T_P$. This means $T_P$
can be decomposed into a sum of $n$ $1$-dimensional complex vector
spaces:
$$T_P=\bigoplus_{i=1}^nL_P(\lambda_i),
\qquad
\big(L_P(\lambda_i)\cong\mathbb{C},~\lambda_i\in\mathbb{R}-\{0\}\big),$$
such that the eigenvalue of the skew-Hermitian transformation on
$L_P(\lambda_i)$ is $\sqrt{-1}\lambda_i$. Or equivalently, the
eigenvalue of the action induced by the one-parameter group
$\textrm{exp}(tA)$ on $L_P(\lambda_i)$ is
$\textrm{exp}(\sqrt{-1}\lambda_it)$. Note that these nonzero real
numbers $\lambda_1,\ldots,\lambda_{n}$ are counted with
multiplicities and thus not necessarily mutually distinct. Of course
they depend on the choice of $P$ in $\text{zero}(A)$ but are
independent of the choice of the almost-Hermitian metric which $A$
preserves.

The following residue formula for the $\chi_y$-genus of $(M^n,g,J)$,
which is a beautiful application of the Atiyah-Bott fixed point
formula, is essentially due to Kosniowski (\cite[Theorem 1]{Ko}).

\begin{theorem}[Residue formula for the $\chi_y$-genus]
\label{localizationchiy} With the above notation and symbols
understood, we have \be\label{localizationchiy2}\chi_y(M)=
\sum_{P\in\text{zero}(A)}(-y)^{d_P}=\sum_{P\in\text{zero}(A)}(-y)^{n-d_P}\ee
where $d_P$ denotes the number of \emph{negative} numbers among
$\lambda_1,\ldots,\lambda_{n}$ and the sum is over all the points in
$\text{zero}(A)$.
\end{theorem}

\begin{remark}
~ \begin{enumerate} \item For complex manifolds this result was
discovered by Kosniowski in \cite[Theorem 1]{Ko}, whose proof is an
application of the Atiyah-Bott fixed point formula to the Dolbeault
complex (\ref{DC}) originating from \cite{Lu}. Indeed, if we replace
the Dolbeault complex (\ref{DC}) by the two-step elliptic complex
(\ref{GDC}), this result still holds for almost-complex manifolds,
which has been carried out by the author in \cite{Li} and used to
give some related applications to symplectic geometry. In \cite{Li2}
this idea was further extended .

\item
In the theorem above, the condition that $A$ preserve some
almost-Hermitian metric on $(M,J)$ can be relaxed to assume only
that the endomorphism induced by $A$ on $T_P$ is nonsingular for any
$P\in\text{zero}(A)$, which is what \cite{Ko} treated and the
condition of which is called \emph{``isolated simple zero points"}
in \cite{Ko} . However, if the zero point set $\text{zero}(A)$ is
not necessarily isolated, there is a similar residue formula for
$\chi_y(M)$ which has also been treated in \cite[Theorem 3]{Ko}. But
in this case we need to use the general Lefschetz fixed point
formula of Atiyah-Singer and it needs the additional condition that
$A$ be \emph{compact}. This means the one-parameter group of $A$
lies in a compact group, which is equivalent to the condition that
$A$ preserve an almost-Hermitian metric on $(M,J)$ as we have
required in Theorem \ref{localizationchiy}.
\end{enumerate}
\end{remark}

\subsection{Morse identity for Hamiltonian torus actions}\label{section4.2}
We now turn to the discussion on the Poincar\'{e} polynomial. Unlike
the case of the $\chi_y$-genus in Theorem \ref{localizationchiy},
even if a compact almost-complex manifold $(M,J)$ admits a
compatible vector field $A$, in \emph{general} there is no residue
formula expressing the Betti numbers of $(M,J)$ in terms of the
local information around $\text{zero}(A)$. The reason for the
existence of Theorem \ref{localizationchiy} is that the coefficients
$\chi^p$ in $\chi_y(M)$ are indices of some natural elliptic
operators (\ref{GDC}) and so we can apply the Lefschetz-type fixed
point formula of elliptic complexes developed by Atiyah, Bott and
Singer to the vector field $A$ to obtain Theorem
\ref{localizationchiy}. So the lack of an analogue to Theorem
\ref{localizationchiy} for the Poincar\'{e} polynomial lies in the
fact that in general the Betti numbers can \emph{not} be realized as
indices of some natural elliptic operators. Nevertheless, if we
impose more structures on $(M,J)$ and $A$, we can reduce the
calculations of the Betti numbers of $(M,J)$ to those of
$\text{zero}(A)$ via the Morse equality for \emph{perfect Morse
functions}.

Suppose in this subsection that $(M,\omega)$ is a compact connected
symplectic manifold of real dimension $2n$ and with a symplectic
form $\omega$. All relevant facts mentioned in what follows can be
found in three excellent books: \cite[Ch. 4]{Au}, \cite[\S 5.5]{MS}
or \cite[\S 3.6]{Ni}. We call a torus action ($T$-action) on
$(M^{n},\omega)$ \emph{symplectic} if this action preserves the
symplectic form $\omega$. We choose a generating vector field $A$
for this $T$-action, i.e., the one-parameter group of $A$ is dense
in this torus $T$. Note that in this case the $T$-action on
$(M^{n},\omega)$ is symplectic if and only if the one form
$i_A(\omega):=\omega(A,\cdot)$ is closed, where $i_A(\cdot)$ denotes
the contraction operator with respect to $A$. Indeed, we know from
the definition of $A$ that the $T$-action on $(M^{n},\omega)$ is
symplectic if and only if the Lie derivative $L_A(\omega)=0$. The
Cartan formula $L_A=\text{d}\circ i_A+i_A\circ\text{d}$ and the
closedness of $\omega$ tell us that $L_A(\omega)=0$ is equivalent to
$\text{d}\big(i_A(\omega)\big)=0$. We call this $T$-action
\emph{Hamiltonian} if the one form $i_A(\omega)$ is exact. This
means there exists a function $f$ on $M$, which is called the
\emph{moment map} of this $T$-action and is unique up to an additive
constant, such that $i_A(\omega)=\text{d}f.$ It is well-known that
this $f$ is a \emph{perfect Morse-Bott function} and
$\text{Crit}(f)$, the critical point set of $f$, coincides with
$\text{zero}(A)$. The latter also coincides with the fixed point set
of the $T$-action.

Note also that we can choose an almost-complex structure $J$ on $M$
such that it is both compatible with $\omega$ and preserved by this
$T$-action (\cite[Lemma 5.52]{MS}). The compatibility between
$\omega$ and $J$ tells us that the bilinear form
$g(v,w):=\omega(v,Jw)$ is an almost-Hermitian metric on $M$. The
facts that $T$-action preserve $\omega$ and $J$ imply that this
$T$-action preserves the metric $g$:
$$\forall t\in T,~t^{\ast}(g)(v,w)=g(t_{\ast}v,t_{\ast}w)=
\omega(t_{\ast}v,Jt_{\ast}w)=\omega(t_{\ast}v,t_{\ast}Jw)=
\omega(v,Jw)=g(v,w).$$ This means the vector field $A$  preserves
both the almost-complex structure $J$ and the almost-Hermitian
metric $g$. Now we assume further that the fixed points of this
$T$-action, which coincides with $\text{zero}(A)=\text{Crit}(f)$,
are all isolated. In this case the above-mentioned function $f$
degenerates to a perfect Morse function and, at each isolated point
$P\in\text{zero}(A)=\text{Crit}(f)$, the Morse index of $f$ is
$2d_P$, twice the number of negative numbers among
$\lambda_1,\ldots,\lambda_n$. Here we use the notation and symbols
introduced in the last subsection. Then the Morse-type equality for
this perfect Morse function $f$ yields the following result.

\begin{theorem}[Residue formula for the Poincar\'{e} polynomial]
\label{localizationpoincare} Suppose a compact connected symplectic
manifold $(M^n,\omega)$ admits a Hamiltonian torus action with
isolated fixed point set. With the above-discussed notions and
symbols understood, the Morse-type equality for the perfect Morse
function $f$ gives us \be\label{localizationpoincare2}
P_y(M)=\sum_{P\in\text{zero}(A)}y^{2d_P}=
\sum_{P\in\text{zero}(A)}y^{2(n-d_P)},\ee
 where the sum is over all
the isolated points in $\text{Crit}(f)=\text{zero}(A)$. Moreover,
(\ref{localizationpoincare2}) still holds if we replace $P_y(M)$
with the Poincar\'{e} polynomial with respect to any coefficient
field $K$, $P_y(M;K)$, where
$$P_y(M;K):=\sum_{i}b_i(M;K)y^i$$
 and $b_i(M;K)$ is
the $i$-th Betti number with respect to the field $K$. Consequently,
the integral homology $H_{\ast}(M,\mathbb{Z})$ of $M$  has no
odd-dimensional homology and is torsion-free.
\end{theorem}

\begin{remark}
When the almost-complex structure $J$ is integral, i.e., $M$ is
K\"{a}hler, Theorem \ref{localizationpoincare} is a classical result
due to Frankel in \cite[\S 4, Corollary 2]{Fr}.
\end{remark}

\subsection{Applications}\label{section4.3}
In this subsection we give an application via the two residue
formulae discussed above.

Our application here is concerned with various geometric and
topological obstructions to the existence of Hamiltonian torus
actions on compact connected symplectic manifold with isolated fixed
points, and is motivated by a famous conjecture in symplectic
geometry, which was raised by McDuff in her seminal paper \cite{Mc}
and now is commonly called the \emph{McDuff conjecture} or
\emph{Frankel-Mcduff conjecture} as \cite{Mc} is inspired by
Frankel's another seminal paper \cite{Fr}. Suppose we have a compact
connected symplectic manifold equipped with a symplectic torus
action. In symplectic geometry it is an important topic to detect
whether or not this given symplectic torus action is Hamiltonian
(\cite[Ch 5]{MS}). The famous McDuff conjecture, which is still open
in its generality, says that any symplectic circle action with
\emph{isolated} fixed points must be Hamiltonian (\cite{Mc}). Many
partial results towards this conjecture have been obtained over the
past two decades (see \cite[Introduction]{Li} and the references
therein). In \cite{Li}, the author applies the rigidity property of
the elliptic operators (\ref{GDC}) to give a criterion to detect if
a given symplectic circle action with isolated fixed points is
Hamiltonian. By using this criterion we can both recover all the
previously known results towards this conjecture and simplify their
proofs.

Roughly speaking, our next application, Theorem \ref{obstruction},
attempts to explain that the existence of compact symplectic
manifolds equipped with Hamiltonian torus actions with isolated
fixed points is a very ``rare" phenomenon via finding out as many
geometric and topological obstructions as possible imposed on these
symplectic manifolds. The main strategy of this application is to
employ the two residue formulas (\ref{localizationchiy2}) and
(\ref{localizationpoincare2}) simultaneously.

\begin{theorem}\label{obstruction}
{\rm{If a compact connected symplectic manifold $(M^{2n},\omega)$
admits a Hamiltonian torus action with isolated fixed points, then

\begin{enumerate}
\item
The integral homology $H_{\ast}(M,\mathbb{Z})$ of $M$ is
torsion-free and has no odd-dimensional homology.

\item
$$\chi_{-y^2}(M)=P_y(M).$$
This means that the $\chi_y$-genus and the Poincar\'{e} polynomial
are essentially the same.

\item
The signature of $M$ is equal to
$$\sum_{i=0}^{n}(-1)^ib_{2i}(M),$$
the alternating sum of its even-dimensional Betti numbers.

\item
The characteristic numbers
$$c_n,\qquad c_1c_{n-1},\qquad (c_{1}^{2}+3c_{2})c_{n-2}-
(c_{1}^{3}-3c_{1}c_{2}+3c_{3})c_{n-3},\qquad\cdots$$ of $M$ can be
completely determined by Betti numbers of $M$ in a very explicit
manner:
$$c_n=\sum_ib_{2i},\qquad c_1c_{n-1}=6\sum_ii(i-1)b_{2i}-
\frac{n(3n-5)}{2}\sum_ib_{2i},\qquad\cdots.$$

\end{enumerate}}}
\end{theorem}

\begin{proof}
$(1)$ has been mentioned in Theorem \ref{localizationpoincare}. (2)
comes from the two residue formulas (\ref{localizationchiy2}) and
(\ref{localizationpoincare2}). Indeed under our condition
(\ref{localizationchiy2}) and (\ref{localizationpoincare2}) read
$$\chi_y(M)=\sum_{P\in~\textrm{zero}(A)}(-y)^{d_P}\qquad
\textrm{and} \qquad P_y(M)=\sum_{P\in~\textrm{Crit}(f)}y^{2d_P},$$
which yield $(2)$. $(3)$ is a corollary of $(2)$ as
$\chi_y(M)\big|_{y=1}$ is nothing but the signature of $M$. $(4)$ is
a corollary of $(2)$ and the $``-1"$-phenomenon described in Section
\ref{section2.1}.
\end{proof}

\begin{remark}
~\begin{enumerate} \item
 In the theorem above, property $(1)$ should be quite
well-known to experts. But to the author's best knowledge, nobody
states it as explicitly as ours in the previous literature for
compact connected symplectic manifolds with Hamiltonian torus
actions.

\item
In the case of \emph{circle} actions, property $(3)$ has been
obtained by Jones-Rawnsley (\cite{JR}) via the Atiyah-Bott fixed
point formula. Recall that the signature is by definition the index
of the intersection pairing on the middle dimensional cohomology of
$M$ and thus \emph{a priori} depends on the \emph{ring} structure of
$H^{\ast}(M;\mathbb{R})$. However, property $(3)$ reveals an
interesting phenomenon for $M$: its signature depends only on the
\emph{additive} structure of $H^{\ast}(M;\mathbb{R})$.
\end{enumerate}
\end{remark}

\bibliographystyle{amsalpha}

\end{document}